\documentclass[11pt,reqno]{amsart}
\usepackage{indentfirst, amssymb,amsmath,amsthm, mathrsfs,cite,graphicx,float,caption,subcaption}
\newtheorem{theo}{Theorem}[section]
\newtheorem{lem}{Lemma}[section]

\newtheorem{rem}{Remark}[section]

\newtheorem{open problem}{Open problem}[section]

\newcommand{\ol}{\overline}
\newcommand{\be}{\begin{equation}}
\newcommand{\ee}{\end{equation}}
\newcommand{\bs}{\begin{small}}
\newcommand{\es}{\end{small}}
\newcommand{\beas}{\begin{eqnarray*}}
\newcommand{\eeas}{\end{eqnarray*}}
\newcommand{\bea}{\begin{eqnarray}}
\newcommand{\eea}{\end{eqnarray}}
\renewcommand{\epsilon}{\varepsilon}
\numberwithin{equation}{section}

\begin{document}
\title[On certain subclasses]{On certain subclasses of analytic and harmonic mappings}
\author[R. Biswas]{Raju Biswas}
\date{}
\address{Department of Mathematics, Raiganj University, Raiganj, West Bengal-733134, India.}
\email{rajubiswasjanu02@gmail.com}
\maketitle
\let\thefootnote\relax
\begin{abstract} Let $\mathcal{H}$ be the class of harmonic functions $f=h+\ol g$ in the unit disk $\mathbb{D}:=\{z\in\mathbb{C}:|z|<1\}$, where $h$ and $g$ are analytic in 
$\mathbb{D}$ with the normalization $h(0)=g(0)=h'(0)-1=0$. Let $\mathcal{D}_{\mathcal{H}}^0(\alpha, M)$ denote the class of functions $f=h+\ol g\in\mathcal{H}$ satisfying the 
conditions $\left|(1-\alpha)h'(z)+\alpha zh''(z)-1+\alpha\right|\leq M+\left|(1-\alpha)g'(z)+\alpha zg''(z)\right|$ with $g'(0)=0$ for $z\in\mathbb{D}$, $M>0$ and $\alpha\in(0,1]$. In this paper, we investigate fundamental properties for functions in 
the class $\mathcal{D}_{\mathcal{H}}^0(\alpha, M)$, such as the coefficient bounds, 
growth estimates, starlikeness and some other properties. Furthermore, we obtain the sharp bound of the second Hankel determinant of inverse logarithmic coefficients for normalized analytic univalent functions $f\in\mathcal{P}(M)$ in $\mathbb{D}$ satisfying the condition $\text{Re}\left(zf''(z)\right)>-M$ for $0<M\leq 1/\log4$ and $z\in\mathbb{D}$.
\end{abstract}
\section{introduction}
\noindent Harmonic mappings are a useful tool in the study of fluid flow problems (see \cite{AC2014}). In addition, planar fluid dynamics problems naturally give rise to univalent harmonic functions 
with special geometric properties such as convexity, starlikeness and close-to-convexity. Univalent harmonic functions are also used in the representation of minimal surfaces. For 
example, Heinz \cite{H1} used such mappings in the study of the Gaussian
curvature of nonparametric minimal surfaces over the unit disc (see \cite[p. 182, section 10.3]{24P}) and Aleman {\it et al.} \cite[Theorem 4.5]{AC2014} considered a fluid flow problem on a convex domain $\Omega$ satisfying an interesting geometric property. After this brief motivation, we will now focus on univalent harmonic mappings.\\[2mm]
\indent Let $f=u+iv$ be a complex-valued function of $z=x+i y$ in a simply connected domain $\Omega$. If $f\in C^2(\Omega)$ (continuous first and second partial derivatives in $\Omega$) and  satisfies the Laplace equation $\Delta f =4f_{z\ol z} = 0$ in $\Omega$, 
then $f$ is said to be harmonic in $\Omega$.  Note that every harmonic mapping $f$ has the canonical representation $f = h + \ol g$, where $h$ and 
$g$ are analytic in $\Omega$, known respectively as the analytic and co-analytic parts of $f$, and $\ol{g(z)}$ denotes the complex conjugate of $g(z)$. The Jacobian of $f$ is defined by $J_f(z):=|h'(z)|^2-|g'(z)|^2$. The inverse function theorem and a result of Lewy \cite{L1936} shows that a harmonic function $f$ is locally univalent in 
$\Omega$ if, and only if, the Jacobian  of $f$, defined by $J_f(z):=|h'(z)|^2-|g'(z)|^2$ is non-zero in $\Omega$.  A locally univalent harmonic function $f$ is said to be sense-preserving if $J_f(z)>0$  in $D$ and sense-reversing if $J_f(z)<0$ in $D$ (see \cite{17P,4P,24P,V1}). Let $\mathcal{H}$ be the class of all complex-valued harmonic 
functions $f=h+\ol g$ defined in $\mathbb{D}$, where $h$ and $g$ are analytic in $\mathbb{D}$ with the normalization $h(0)=h'(0)-1=0$ and $g(0)=0$. If the co-analytic part $g(z)\equiv 0$ 
in $\mathbb{D}$, then the class $\mathcal{H}$ reduces to the class $\mathcal{A}$ of analytic functions in $\mathbb{D}$ with $f(0)=0$ and $f'(0)=1$.
Let $\mathcal{S}_{\mathcal{H}}$ denote the subclass of $\mathcal{H}$ that are sense-preserving and univalent in $\mathbb{D}$ and let 
$\mathcal{S}_{\mathcal{H}}^0=\left\{f=h+\ol g\in\mathcal{S}_{\mathcal{H}}: g'(0)=0 \right\}$. The analytic and co-analytic parts of every $f=h+\ol g\in\mathcal{S}_{\mathcal{H}}^0$ have the following forms:
\bea\label{re1}h(z)=z+\sum_{n=2}^\infty a_nz^n\;\;\text{and}\;\;g(z)=\sum_{n=2}^\infty b_nz^n. \eea
If $g(z)\equiv 0$ in $\mathbb{D}$, then both the classes $\mathcal{S}_{\mathcal{H}}$ and $\mathcal{S}_{\mathcal{H}}^0$ reduces to the class 
$\mathcal{S}$ of analytic and univalent functions in $\mathbb{D}$ with $f(0)=f'(0)-1=0$. Both $\mathcal{S}_{\mathcal{H}}$ and $\mathcal{S}_{\mathcal{H}}^0$ are natural harmonic 
generalizations of $\mathcal{S}$, but only $\mathcal{S}_{\mathcal{H}}^0$ is known to be compact although both $\mathcal{S}_{\mathcal{H}}$ and $\mathcal{S}_{\mathcal{H}}^0$ are normal. In 1984, Clunie and Sheil-Small \cite{4P} undertook a comprehensive study of the class $\mathcal{S}_{\mathcal{H}}$ and its geometric subclasses. This study 
has subsequently garnered extensive attention from researchers (see \cite{VV1,V4,3P,7P,15P, MBG2024, BM2025}). \\[2mm]
\indent  A domain $\Omega$ is called starlike with respect to a point $z_0\in\Omega$ if the line segment joining $z_0$ to any point in $\Omega$ lies in $\Omega$. In particular, if 
$z_0=0$, then $\Omega$ is simply called starlike. A complex-valued harmonic mapping $f\in\mathcal{H}$ is said to be starlike if $f(\mathbb{D})$ is starlike. We denote the class of 
harmonic starlike functions in $\mathbb{D}$ by $\mathcal{S}_{\mathcal{H}}^*$. A domain $\Omega$ is called convex if it is starlike with respect to every point in $\Omega$. A function 
$f\in\mathcal{H}$ is said to be convex if $f(\mathbb{D})$ is convex. The class of all harmonic convex mappings in $\mathbb{D}$ is denoted by $\mathcal{K}_{\mathcal{H}}$.
Starlikeness is a hereditary property for conformal mappings.
Thus if $f$ is analytic and univalent in $\mathbb{D}$ with $f(0)=0$ and if $f$ maps $\mathbb{D}$ onto a domain that is starlike with respect to the origin, then the image of every subdisk
$|z|<r<1$ is also starlike with respect to the origin. 
Again, this 
hereditary property does not generalize to harmonic mappings, which is being discussed in \cite{17P}.\\[2mm] 
 Let $\mathcal{R}$ be the class of all analytic functions $h$ in $\mathbb{D}$ such that $h(0)=h'(0)-1=0$ and $\textrm{Re}\left(h'(z)\right)>0$ in $\mathbb{D}$.
It is well-known that $\mathcal{R}\subsetneq\mathcal{S}$. MacGregor \cite{M1962} proved that if $h\in\mathcal{R}$, then each partial sum $s_n(h)=\sum_{k=0}^n a_k z^k$ is 
univalent in $|z|<1/2$ for $n\geq 2$ and $h(z)$ maps the disk $|z|<\sqrt{2}-1$ onto a convex domain. The numbers $1/2$ and $\sqrt{2}-1$ are the best possible constants. In 
\cite{S1970}, Singh proved that if $h\in\mathcal{R}$, then each partial sum $s_n(h)$ is convex in $|z| < 1/4$ and the number $1/4$ is the best possible constant. \\[2mm]
\indent In 2013, Ponnusamy {\it et al.} \cite{PYY2013} studied the following class as a harmonic analog of the class $\mathcal{R}$: 
\beas\mathcal{P}_{\mathcal{H}}:=\left\{f=h+\ol{g}\in\mathcal{H}: \textrm{Re}\left(h'(z)\right)>|g'(z)|\quad\text{in}\quad\mathbb{D}\right\} \eeas
and $\mathcal{P}_{\mathcal{H}}^0:=\left\{f=h+\ol{g}\in\mathcal{P}_{\mathcal{H}}: g'(0)=0\right\}$. The authors of \cite{PYY2013} proved that functions in $\mathcal{P}_{\mathcal{H}}$ are close-to-convex in $\mathbb{D}$. In \cite{1LP2013} and \cite{2LP2013}, Li and Ponnusamy have investigated the radius of univalency and convexity of sections of functions $f\in \mathcal{P}_{\mathcal{H}}^0$, respectively.\\[2mm]
\indent In 2020, Ghosh and Allu \cite{16P} established the coefficient bound problem and the growth theorem for functions in the class
\beas\mathcal{P}_{\mathcal{H}}^0(M)=\{h+\ol{g}\in\mathcal{H}: \text{Re}\left(zh''(z)\right)>-M+|zg''(z)|\;\text{with}\; g'(0)=0\;\text{for}\;M>0, z\in\mathbb{D}\}\eeas
and a two-point distortion theorem for functions in the class
\beas\mathcal{B}_{\mathcal{H}}^0(M)=\{h+\ol{g}\in\mathcal{H}: \left|zh''(z)\right|\leq M-|zg''(z)|\;\text{with}\; g'(0)=0\;\text{for}\;M>0, z\in\mathbb{D}\}.\eeas
The subclasses $\mathcal{B}_{\mathcal{H}}^0(M)$ and $\mathcal{P}_{\mathcal{H}}^0(M)$ are not only the generalizations of analytic functions but also they are closely related to the 
analytic subclasses $\mathcal{B}(M)$ and $\mathcal{P}(M)$ respectively and the classes are defined by
\bea\label{k1}\left\{\begin{array}{lll}
\mathcal{P}(M)=\{h\in\mathcal{A}: \text{Re}\left(zh''(z)\right)>-M\;\text{for}\;M>0, z\in\mathbb{D}\},\\[2mm]
\mathcal{B}(M)=\{h\in\mathcal{A}: \left|zh''(z)\right|\leq M\;\text{for}\;M>0, z\in\mathbb{D}\}.\end{array}\right. \eea
The classes mentioned in (\ref{k1}) have been studied by Mocanu \cite{334}, and Ponnusamy and Singh \cite{332}. In 1995, Ali {\it et al.}\cite{333} proved that each function in the class $\mathcal{P}(M)$ 
is univalent and starlike in the unit disk $\mathbb{D}$ for $0<M\leq 1/\log 4 (\approx 0.7213475)$. Afterwards, Ponnusamy and Singh \cite{332} showed that each function in the class $\mathcal{B}(M)$ are univalent and starlike whenever $0<M\leq 1$ and convex whenever $0<M\leq 1/2$.\\[2mm] 
\indent Motivated by the results of \cite{PYY2013, 1LP2013, 2LP2013} and the class $\mathcal{B}_{\mathcal{H}}^0(M)$, in this paper,
we consider the class $\mathcal{D}_{\mathcal{H}}^0(\alpha, M)$ of all functions $f=h+\ol{g}\in\mathcal{H}$ for $M>0, \alpha\in(0,1]$ that satisfy the following conditions:
\beas\left|(1-\alpha)h'(z)+z\alpha h''(z)-(1-\alpha)\right|\leq M-\left|(1-\alpha)g'(z)+\alpha zg''(z)\right|\quad\text{with}~ g'(0)=0\eeas
for $z\in\mathbb{D}$. It is evident that $\mathcal{D}_{\mathcal{H}}^0(1, M)=\mathcal{B}_{\mathcal{H}}^0(M)$.\\
\noindent The organization of this paper is: In section $2$, we establish the sharp coefficients bounds, growth results, starlikeness and some other properties for functions 
in $\mathcal{D}_{\mathcal{H}}^0(\alpha, M)$. In section $5$, we obtain the sharp bound for the second Hankel determinant of logarithmic inverse coefficients for functions in the class 
$\mathcal{P}(M)$. The remaining sections contain introductions and key lemmas.
\section{fundamental properties }
In the following result, we obtain the sharp coefficient bounds for functions in the class $\mathcal{D}_{\mathcal{H}}^0(\alpha, M)$.
\begin{theo}\label{T2} Let $M> 0$, $\alpha\in(0, 1]$ and $f=h+\ol{g}\in\mathcal{D}_{\mathcal{H}}^0(\alpha, M)$  be of the form (\ref{re1}). For $n\geq 2$, we have $|a_n|\leq M/\left(n+(n^2-2n)\alpha\right)$ and $|b_n|\leq M/\left(n+(n^2-2n)\alpha\right)$.
The result is sharp for the functions $f_1$ and $f_2$, where the functions are given by $f_1(z)=z+Mz^n/\left(n+(n^2-2n)\alpha\right)$ and $f_2(z)=z+M\ol{z^n}/\left(n+(n^2-2n)\alpha\right)$ for $n\geq 2$. \end{theo}
\begin{proof} As $f=h+\ol{g}\in \mathcal{D}_{\mathcal{H}}^0(\alpha, M)$, we have 
\bea\label{re2}  \left|(1-\alpha)h'(z)+z\alpha h''(z)-(1-\alpha)\right|\leq M -\left|(1-\alpha)g'(z)+\alpha zg''(z)\right|~\text{for}~ z\in\mathbb{D}. \eea
Since $(1-\alpha)h'(z)+z\alpha h''(z)-(1-\alpha)=\sum_{n=2}^\infty \left(n+(n^2-2n)\alpha\right)a_nz^{n-1}$ is analytic in $\mathbb{D}$, then in view of Cauchy's integral formula for derivatives, we have 
\beas \left(n+(n^2-2n)\alpha\right) a_n=\frac{1}{2\pi i}\int_{|z|=r}\frac{(1-\alpha)h'(z)+z\alpha h''(z)-(1-\alpha)}{z^{n}}dz.\eeas
Therefore, we have
\beas \left(n+(n^2-2n)\alpha\right) |a_n|&=&\left|\frac{1}{2\pi i}\int_0^{2\pi}\frac{(1-\alpha)h'(re^{i\theta})+\alpha re^{i\theta} h''(re^{i\theta})-(1-\alpha)}{r^{n}e^{in\theta}}ir e^{i\theta}d\theta\right|\\[2mm]
&\leq&\frac{1}{2\pi}\int_0^{2\pi}\frac{\left|(1-\alpha)h'(re^{i\theta})+\alpha re^{i\theta}h''(re^{i\theta})-(1-\alpha)\right|}{r^{n-1}}d\theta.\eeas
From (\ref{re2}), we have 
\beas \left(n+(n^2-2n)\alpha\right) r^{n-1}|a_n|&\leq&\frac{1}{2\pi}\int_0^{2\pi}\left(M-\left|(1-\alpha)g'(re^{i\theta})+\alpha re^{i\theta}g''(re^{i\theta})\right|\right)d\theta\\[2mm]
&\leq&M-\left|\frac{1}{2\pi}\int_0^{2\pi}\left(g'(re^{i\theta})+re^{i\theta}g''(re^{i\theta})\right)d\theta\right|=M.\eeas
Letting $r\to 1^-$ gives the desired bound $|a_n|\leq M/\left(n+(n^2-2n)\alpha\right)$. Using similar argument as above, we obtain $|b_n|\leq M/\left(n+(n^2-2n)\alpha\right)$ for $n\geq 2$. It is evident that 
$f_1(z)=z+Mz^n/\left(n+(n^2-2n)\alpha\right)$ and $f_2(z)=z+M\ol{z^n}/\left(n+(n^2-2n)\alpha\right)$ $(n\geq2)$ are in the class $\mathcal{D}_{\mathcal{H}}^0(\alpha, M)$ with $|a_n(f_1)|=M/\left(n+(n^2-2n)\alpha\right)$ and $|b_n(f_2)|=M/\left(n+(n^2-2n)\alpha\right)$. 
This completes the proof.
\end{proof}
\begin{rem} Setting $\alpha=1$ in \textrm{Theorem \ref{T2}} gives \textrm{Theorem 2.2} of \cite{GA2017}.\end{rem}
\noindent Let us consider the class $\mathcal{D}(\alpha, M)$ of all functions $\phi\in \mathcal{A}$ satisfying the following condition:
\beas \left|(1-\alpha)\phi'(z)+\alpha z\phi''(z)-(1-\alpha)\right|\leq M\quad \text{for}\quad M >0, \;\alpha\in(0, 1]\quad \text{and}\quad z\in\mathbb{D}.\eeas
The following result gives a correlation between the functions in the classes $\mathcal{D}(\alpha, M)$ and $\mathcal{D}_{\mathcal{H}}^0(\alpha, M)$.
\begin{theo}\label{T1}
The harmonic map $f=h+\ol{g}$ belongs to $\mathcal{D}_{\mathcal{H}}^0(\alpha, M)$ if, and only if, the
function $F_\epsilon=h +\epsilon g$ belongs to $\mathcal{D}(\alpha, M$) for each $\epsilon$ with $|\epsilon|=1$.
\end{theo}
\begin{proof} Let $f=h+\ol{g}\in \mathcal{D}_{\mathcal{H}}^0(\alpha, M)$. Therefore, 
\beas \left|(1-\alpha)h'(z)+\alpha zh''(z)-(1-\alpha)\right|\leq M-\left|(1-\alpha)g'(z)+\alpha zg''(z)\right|\quad\text{for}\quad z\in\mathbb{D}. \eeas
Fix $|\epsilon|=1$. Since $F_\epsilon=h +\epsilon g$, thus, we have
\beas &&\left|(1-\alpha)F_{\epsilon}'(z)+\alpha zF_{\epsilon}''(z)-(1-\alpha)\right|\\[2mm]
&=&\left|\left((1-\alpha)h'(z)+\alpha zh''(z)-(1-\alpha)\right)+\epsilon\left((1-\alpha)g'(z)+\alpha zg''(z)\right)\right|\\[2mm]
&\leq&\left|(1-\alpha)h'(z)+\alpha zh''(z)-(1-\alpha)\right|+\left|(1-\alpha)g'(z)+\alpha zg''(z)\right|\leq M\;\;\text{for} \;z\in\mathbb{D},\eeas
which shows that $F_\epsilon=h +\epsilon g\in\mathcal{D}(\alpha, M)$ for each $\epsilon$ with $|\epsilon|=1$. Conversely, if $F_\epsilon\in\mathcal{D}(\alpha, M)$, for $z\in\mathbb{D}$, we have
\beas&& \left|(1-\alpha)F_{\epsilon}'(z)+\alpha zF_{\epsilon}''(z)-(1-\alpha)\right|\leq M,\\[1mm]\text{\it i.e.,}
&&\left|\left((1-\alpha)h'(z)+\alpha zh''(z)-(1-\alpha)\right)+\epsilon\left((1-\alpha)g'(z)+\alpha zg''(z)\right)\right|\leq M.\eeas
Since $\epsilon$ $(|\epsilon|=1)$ is arbitrary, for an appropriate choice of $\epsilon$, we have
\beas \left|(1-\alpha)h'(z)+\alpha zh''(z)-(1-\alpha)\right|+\left|(1-\alpha)g'(z)+\alpha zg''(z)\right|\leq M\;\;\text{for} \;z\in\mathbb{D}, \eeas
which shows that $f\in\mathcal{D}_{\mathcal{H}}^0(\alpha, M)$. This completes the proof.
 \end{proof}
In the following result, we establish the sharp growth estimates for functions in the class $\mathcal{D}_{\mathcal{H}}^0(\alpha, M)$. 
\begin{theo}\label{T3} Let $M>0$, $\alpha\in(0, 1]$ and $f=h+\ol{g}\in \mathcal{D}_{\mathcal{H}}^0(\alpha, M)$ be of the form (\ref{re1}). Then, 
\bea\label{req1}|z|-\frac{M|z|^2}{2}\leq |f(z)|\leq |z|+\frac{M|z|^2}{2}.\eea
For each $z\in\mathbb{D}$, $z\not=0$, equality occurs for the function $f$ given by $f(z)=z+Mz^2/2$ or its suitable rotations.\end{theo}
\begin{proof} Let $f=h+\ol{g}\in \mathcal{D}_{\mathcal{H}}^0(\alpha, M)$. In view of \textrm{Theorem \ref{T1}}, we have $F_{\epsilon}=h+\epsilon g\in \mathcal{D}(\alpha, M)$ for each $|\epsilon|=1$. For $z\in\mathbb{D}$, we have
\beas&& \left|(1-\alpha)F_{\epsilon}'(z)+\alpha zF_{\epsilon}''(z)-(1-\alpha)\right|\\[2mm]
&=&\left|\left((1-\alpha)h'(z)+\alpha zh''(z)-(1-\alpha)\right)+\epsilon\left((1-\alpha)g'(z)+\alpha zg''(z)\right)\right|\leq M. \eeas
Thus, according to the subordination principle, there exists an analytic function $\omega : \mathbb{D}\to \mathbb{D}$ with $\omega(0)=0$ such that
\bea&& (1-\alpha)F_{\epsilon}'(z)+\alpha zF_{\epsilon}''(z)-(1-\alpha)=M\omega(z),\nonumber\\[2mm]\text{\it i.e.,}
\label{re5}&&\frac{d}{dz}\left(\alpha z^{1/\alpha -1} F_\epsilon'(z)\right)=M z^{1/\alpha -2}\omega(z)+(1-\alpha)z^{1/\alpha -2}.\eea
In view of the Schwarz lemma, we have $|\omega(z)|\leq|z|$ for $z\in\mathbb{D}$. 
Note that $z^{a}=\exp(a\log(z))$, where $a>0$ and the branch of the logarithm is determined by $\log(1)=0$. This guarantees that the function is both single-valued and analytic within that range. Let us consider two cases.\\
{\bf Case 1.} Let $\alpha\not=1$.
Using $F_{\epsilon}'(0)=1$, from (\ref{re5}), we have 
\bea\label{a1}&& \left|\alpha z^{1/\alpha -1} F_\epsilon'(z)\right|\nonumber\\
&=&\left|(1-\alpha)\int_{0}^{|z|}(te^{i\theta})^{1/\alpha -2}e^{i\theta}dt+M\int_{0}^{|z|}(te^{i\theta})^{1/\alpha -2} \omega (te^{i\theta}) e^{i\theta}dt\right|\\[2mm]
&\leq& (1-\alpha)\int_{0}^{|z|}t^{1/\alpha -2}dt+M\int_{0}^{|z|} t^{1/\alpha -1} dt\nonumber\\[2mm]
&=&\alpha |z|^{1/\alpha -1}+M\alpha |z|^{1/\alpha}.\nonumber\eea
Therefore, we have 
\bea\label{re66}|F_{\epsilon}'(z)|=|h'(z)+\epsilon g'(z)|\leq 1+M|z|.\eea 
Since $\epsilon$ $(|\epsilon|=1)$ is arbitrary, it follows from (\ref{re66}) that $|h'(z)|+|g'(z)|\leq 1+M |z|$.
Let $\Gamma$ be the radial segment from $0$ to $z$. Therefore,
\beas|f(z)|&=&\left|\int_{\Gamma}\left(\frac{\partial f}{\partial \xi}d\xi+\frac{\partial f}{\partial \ol{\xi}}d\ol{\xi}\right)\right|\\
&\leq& \int_{\Gamma}\left(|h'(\xi)|+|g'(\xi)|\right)|d\xi|\leq \int_0^{|z|}\left(1+M t\right)dt=|z|+M\frac{|z|^2}{2}.\eeas
From (\ref{a1}), we have 
\bea\label{ew22}\left|\alpha z^{1/\alpha -1} F_\epsilon'(z)\right|
&=&\left|(1-\alpha)\int_{0}^{|z|}(te^{i\theta})^{1/\alpha -2}e^{i\theta}dt+M\int_{0}^{|z|}(te^{i\theta})^{1/\alpha -2} \omega (te^{i\theta}) e^{i\theta}dt\right|\nonumber\\
&\geq &(1-\alpha)\int_{0}^{|z|}t^{1/\alpha -2}dt+M\int_{0}^{|z|}t^{1/\alpha -2} \text{Re}\left(\omega (te^{i\theta})\right) dt \nonumber\\
&\geq& \alpha |z|^{1/\alpha -1}-M\int_{0}^{|z|}t^{1/\alpha -1} dt=\alpha |z|^{1/\alpha -1}-M\alpha |z|^{1/\alpha}.\eea
From (\ref{ew22}), we obtain
\bea\label{re88} \left|F_{\epsilon}'(z)\right|=\left|h'(z)+\epsilon g'(z)\right|\geq 1-M |z|.\eea
Since $\epsilon$ $(|\epsilon|=1)$ is arbitrary, it follows from (\ref{re88}) that
\bea\label{re99} |h'(z)|-|g'(z)|\geq 1-M|z|.\eea
In view of (\ref{re99}), we obtain
\beas|f(z)|=\left|\int_0^z\left(\frac{\partial f}{\partial \xi}d\xi+\frac{\partial f}{\partial \ol{\xi}}d\ol{\xi}\right)\right|&\geq& \int_0^{|z|}\left(|h'(\xi)|-|g'(\xi)|\right)|d\xi|\\
&\geq&|z|-M\int_0^{|z|} tdt=|z|-M\frac{|z|^2}{2}.\eeas
{\bf Case 2.} Let $\alpha=1$. From (\ref{re5}), we have 
\beas \frac{d}{dz}\left(F_\epsilon'(z)\right)=\frac{M\omega(z)}{z}.\eeas 
Using $F_{\epsilon}'(0)=1$, we have
\beas\left|F_{\epsilon}'(z)\right|=\left|1+M\int_{0}^{|z|} \frac{\omega(te^{i\theta})}{te^{i\theta}} e^{i\theta}dt\right|\leq 1+M|z|.\eeas
Similarly, we have 
\beas\left|F_{\epsilon}'(z)\right|=\left|1+M\int_{0}^{|z|} \frac{\omega(te^{i\theta})}{te^{i\theta}} e^{i\theta}dt\right|\geq 1+M\int_{0}^{|z|} \frac{\text{Re}(\omega(te^{i\theta}))}{t}dt\geq 1-M|z|.\eeas
Using the same argument as in Case 1, we arrive at the following conclusion
\beas |z|-M\frac{|z|^2}{2}\leq |f(z)|\leq |z|+M\frac{|z|^2}{2}.\eeas
Equality holds in (\ref{req1}) when the function $f$ given by $f(z)=z+Mz^2/2\in \mathcal{D}_{\mathcal{H}}^0(\alpha, M)$ or its  suitable rotations. This completes the proof.
\end{proof}
\begin{rem} Setting $\alpha=1$ in \textrm{Theorem \ref{T3}} gives \textrm{Theorem 2.3} of \cite{GA2017}.\end{rem}
\noindent In the following result, we establish the upper bound of the Jacobian for functions in the class $\mathcal{D}_{\mathcal{H}}^0(\alpha, M)$. 
\begin{theo} If $f\in\mathcal{D}_{\mathcal{H}}^0(\alpha, M)$ for $M>0$ and $\alpha\in(0, 1]$, then $J_f(z)\leq (1+M|z|)^2$, with equality for the function $f(z)=z+Mz^2/2$.\end{theo}
\begin{proof} As $f=h+\ol{g}\in \mathcal{D}_{\mathcal{H}}^0(\alpha, M)$, thus, we have
\beas\left|(1-\alpha)h'(z)+\alpha zh''(z)-(1-\alpha)\right|\leq M -\left|(1-\alpha)g'(z)+\alpha zg''(z)\right|\leq M~\text{for}~z\in\mathbb{D},\eeas
which shows that $h(z)\in\mathcal{D}(\alpha, M)$. In view of the subordination principle, there exists an analytic function $\omega:\mathbb{D}\to\mathbb{D}$ with $\omega(0)=0$ such that
\beas&& (1-\alpha)h'(z)+\alpha zh''(z)-(1-\alpha)=M\omega(z),\\[2mm]\text{\it i.e.,}
&&\frac{d}{dz}\left(\alpha z^{1/\alpha -1} h'(z)\right)=M z^{1/\alpha -2}\omega(z)+(1-\alpha)z^{1/\alpha -2}.\eeas
Since $h'(0)=1$ and $\omega$ is a Schwarz function, thus, we have $|\omega(z)|\leq|z|$ for $z\in\mathbb{D}$. Utilizing the same argument as in Case 1 and Case 2 of \textrm{Theorem \ref{T3}}, we have $\left| h'(z)\right|\leq 1+M |z|$. Therefore, 
\bea \label{et2}J_f(z)=\left|h'(z)\right|^2-\left|g'(z)\right|^2\leq \left|h'(z)\right|^2\leq \left(1+M|z|\right)^2.\eea
The equality in (\ref{et2}) holds for the function $f=z+Mz^2/2\in\mathcal{D}_{\mathcal{H}}^0(\alpha, M)$. This completes the proof.\end{proof}
\noindent The following theorem gives a sufficient condition for a complex-valued function belonging to $\mathcal{D}_{\mathcal{H}}^0(\alpha, M)$.
\begin{theo}\label{TT1} Let $f=h+\ol{g}\in\mathcal{H}$ with $g'(0)=0$ be given by (\ref{re1}). If 
\bea\label{a2} \sum_{n=2}^\infty \left(n+(n^2-2n)\alpha\right)\left(|a_n|+|b_n|\right)\leq M,\eea
then $f\in\mathcal{D}_{\mathcal{H}}^0(\alpha, M)$.
\end{theo}
\begin{proof}Let $f=h+\ol{g}\in\mathcal{H}$ with $g'(0)=0$ be given by (\ref{re1}). Therefore, we have $h(z)=z+\sum_{n=2}^\infty a_nz^n$ and $g(z)=\sum_{n=2}^\infty b_nz^n$. Using (\ref{a2}), we have 
\beas \left|(1-\alpha)h'(z)+z\alpha h''(z)-(1-\alpha)\right|&=&\left|\sum_{n=2}^\infty \left(n+(n^2-2n)\alpha\right)a_nz^{n-1}\right|\\[1mm]
&\leq& \sum_{n=2}^\infty \left(n+(n^2-2n)\alpha\right)|a_n||z|^{n-1}\\[1mm]
&\leq &\sum_{n=2}^\infty \left(n+(n^2-2n)\alpha\right)|a_n|\\[1mm]
&\leq& M-\sum_{n=2}^\infty\left(n+(n^2-2n)\alpha\right)|b_n|\\[1mm]
&\leq& M-\left|\sum\limits_{n=2}^\infty \left(n+(n^2-2n)\alpha\right) b_n z^{n-1}\right|\\[1mm]
&=&M-\left|(1-\alpha)g'(z)+\alpha zg''(z)\right|,\eeas
which shows that $f\in\mathcal{D}_{\mathcal{H}}^0(\alpha, M)$. This completes the proof.
\end{proof}
\noindent Now, we recall the following known result.
\begin{lem}\cite{20P}\label{lem10} Let $f=h+\ol{g}$ be given by (\ref{e1}). If $\sum_{n=2}^{\infty} n\left(|a_n|+|b_n|\right)\leq 1$, then $f$ is starlike in $\mathbb{D}$.
\end{lem}
\begin{theo} Let $M>0$, $\alpha\in(0,1]$ and $f=h+\ol{g}\in\mathcal{D}_{\mathcal{H}}^0(\alpha, M)$ be given by (\ref{re1}). Then $f$ is starlike in $|z|\leq r_1$,
where $r_1\in(0,1)$ is the smallest root of the equation 
\beas 2Mr \;{}_2F_1\left(1,\frac{1}{\alpha};1+\frac{1}{\alpha}; r\right)-1=0.\eeas\end{theo}
\begin{proof} Let $0<r<1$ and $f_r(z)=f(rz)/r=z+\sum_{n=2}^{\infty} a_nr^{n-1}z^n+\ol{\sum_{n=2}^{\infty} b_nr^{n-1}z^n}$ for 
$z\in\mathbb{D}$.
For convenience, let 
\beas S=\sum_{n=2}^{\infty} n\left(|a_n|+|b_n|\right)r^{n-1}.\eeas
In view of \textrm{Theorem \ref{T2}}, we have 
\beas S\leq 2M\sum_{n=2}^{\infty}\frac{r^{n-1}}{1+(n-2)\alpha}&=&\frac{2M}{\alpha} r^{1-1/\alpha}\sum_{n=2}^{\infty}\int_{\xi=0}^r  \xi^{1/\alpha+n-3}d\xi\\[1mm]
&=&\frac{2M}{\alpha} r^{1-1/\alpha}\int_{\xi=0}^r  \frac{\xi^{1/\alpha-1}}{1-\xi}d\xi\\[1mm]
&=&\frac{2M}{\alpha} r\int_{t=0}^1 \frac{t^{1/\alpha-1}}{1-r t}dt.\eeas
We know that an integral giving the hypergeometric function (see \cite{B1964}) is
\beas {}_2F_1\left(a,b;c; z\right)=\frac{\Gamma(c)}{\Gamma(b)\Gamma(b-c)}\int_{0}^1\frac{t^{b-1}(1-t)^{c-b-1}}{(1-t z)^a}dt.\eeas
Therefore, we have 
\beas S\leq 2Mr \;{}_2F_1\left(1,\frac{1}{\alpha};1+\frac{1}{\alpha}; r\right)\leq 1\quad\text{for}\quad r\leq r_1,\eeas
where $r_1\in(0,1)$ is the smallest root of the equation 
\beas H(r):=2Mr \;{}_2F_1\left(1,\frac{1}{\alpha};1+\frac{1}{\alpha}; r\right)-1=0.\eeas
Note that ${}_2F_1\left(1,1/\alpha;1+1/\alpha; 0\right)=1$, ${}_2F_1\left(1,1/\alpha;1+1/\alpha; 1\right)=+\infty$ and the function $H(r)$ 
is continuous in $[0,1]$ with $\lim_{r\to 0^+} H(r)=-1$ and $\lim_{r\to1^-}H(r)=+\infty$. The intermediate value theorem guarantees the existence of a root for the equation $H(r)=0$ within the interval $(0,1)$.
This completes the proof.\end{proof}
\section{ introduction and preliminaries of Hankel determinants}
\noindent Let $\mathcal{H}_1$ denote the class of analytic functions in the unit disk $\mathbb{D}:=\{z\in\mathbb{C}: |z|<1\}$. Let $\mathcal{A}$ denote the class of functions $f\in\mathcal{H}_1$ such
that $f(0)=0$ and $f'(0)=1$. Let $\mathcal{S}$ denote the subclass of $\mathcal{A}$ such that each functions are univalent in $\mathbb{D}$. If $f\in\mathcal{S}$, then it has the following form:
\be\label{eq1} f(z)=z+\sum_{n=2}^\infty a_n z^n\; \;\text{for}\,z\in\mathbb{D}.\ee
The logarithmic coefficients $\gamma_n$ associated with each $f\in\mathcal{S}$ are defined by
\be\label{eq2} F_f(z):=\log\frac{f(z)}{z}=2\sum_{n=1}^\infty \gamma_n z^n\;\text{for}\;z\in\mathbb{D}.\ee
The logarithmic coefficients $\gamma_n$ are essential in the theory of univalent functions, see \cite[Chapter 5]{8} for more information.
Differentiating (\ref{eq2}) and using (\ref{eq1}), we obtain
\beas \gamma_1=\frac{1}{2}a_2,\; \gamma_2=\frac{1}{2}\left(a_3-\frac{1}{2}a_2^2\right),\;\gamma_3=\frac{1}{2}\left(a_4-a_2a_3+\frac{1}{3}a_2^3\right).\eeas
If $f\in\mathcal{S}$, then by the Bieberbach’s theorem, we have $|a_2|\leq 2$ and hence $|\gamma_1|\leq 1$. Using the Fekete-Szeg\"o
inequality \cite[Theorem 3.8]{8} for functions in $\mathcal{S}$, we obtain $|\gamma_2|=(1/2)\left|a_3-(1/2)a_2^2\right|\leq (1/2)+e^{-2}=0.635...$. For $n\geq3$, the problem 
seems much harder and no significant bound for $|\gamma_n|$ when $f\in\mathcal{S}$. Let $f\in\mathcal{A}$ and $n,q\in\mathbb{N}$.
The Hankel determinants are significant in various areas of study, such as the analysis of singularities \cite[Chapter $X$]{450} and power series with integral coefficients \cite{460}.
For more information on the Hankel determinants, we refer to \cite{440,550}. Let $f\in\mathcal{S}$ and $g=f^{-1}$ be defined in a neighborhood of the origin with the Taylor series expansion
\bea\label{g1} g(\omega)=f^{-1}(\omega)=\omega +\sum_{n=2}^{\infty} A_n\omega^n, \eea 
where we choose $|\omega|<1/4$, as we know from Koebe One-Quarter Theorem (see \cite{8}). L\"owner \cite{34} obtained the sharp bound
$|A_n|\leq 1\cdot3\cdot5\cdots(2n-1)\cdot2^n/(n+1)!$ for $n\geq2$ by using variational method and the equality holds when $f^{-1}$ is the inverse of Koebe function. Equating the coefficients in $f\left(f^{-1}(\omega)\right)=\omega$ by means of (\ref{eq1}) and (\ref{g1}), we derive that
\bea\label{g2} A_2=-a_2,A_3=2a_2^2-a_3,A_3=-5a_2^3+5a_2a_3-a_4,\cdots.\eea
The notion of logarithmic coefficients of the inverse univalent functions was proposed by Ponnusamy {\it et al.} \cite{35}. The logarithmic inverse coefficients $\Gamma_n$ 
($n\in\mathbb{N}$) of $f^{-1}$ are defined by the equation
\be\label{g3} F_{f^{-1}}(\omega):=\log\frac{f^{-1}(\omega)}{\omega}=2\sum_{n=1}^\infty \Gamma_n \omega^n\;\;\text{for}\;\;|\omega|<\frac{1}{4}.\ee
Differentiating (\ref{g3}) together with (\ref{g1}) and (\ref{g2}), we obtain 
\bea \Gamma_1=-\frac{1}{2}a_2,\Gamma_2=-\frac{1}{2}a_3+\frac{3}{4}a_2^2, \Gamma_3=-\frac{1}{2}a_4+2a_2a_3-\frac{5}{3}a_2^3.\eea
In 2018, Ponnusamy {\it et al.} \cite{35} proved that if $f\in\mathcal{S}$, then $|\Gamma_n|\leq(1/(2n))\binom{2n}{n}$, $n\in\mathbb{N}$
and the equality holds only for Koebe function or its rotations.
In 2022, Kowalczyk and Lecko \cite{13} proposed the study of the Hankel determinant whose entries are logarithmic coefficients of $f\in\mathcal{S}$, which is given by
\beas H_{q,n}\left(F_f/2\right):=\left|\begin{array}{ccccc}
\gamma_n&\gamma_{n+1}&\cdots&\gamma_{n+q-1}\\
\gamma_{n+1}&\gamma_{n+2}&\cdots&\gamma_{n+q}\\
\vdots&\vdots&\ddots&\vdots\\
\gamma_{n+q-1}&\gamma_{n+q}&\cdots&\gamma_{n+2(q-1)}\\
\end{array}\right|.\eeas
Also, the authors \cite{13} obtained the sharp bound of second Hankel determinant of $H_{2,1}\left(F_f/2\right)$ for starlike and convex functions.
Numerous authors have extensively investigated the sharp bound of Hankel determinants of logarithmic coefficients, for more details (see \cite{3,RB1,13,14,25,28}).\\[2mm] 
\indent Motivated by the results of \cite{3,RB1,13,14,25,28}, in this paper, we investigate the second Hankel determinant of logarithmic inverse coefficients for functions in the class $\mathcal{P}(M)$ define in (\ref{k1}). Suppose that $f\in\mathcal{S}$ given by (\ref{eq1}). Then the second Hankel determinant of $F_{f^{-1}}/2$ is given by
\bea\label{eq3} H_{2,1}\left(F_{f^{-1}}/2\right)=\Gamma_1\Gamma_3-\Gamma_2^2&=&\frac{1}{4}\left(A_2A_4-A_3^2+\frac{1}{4}A_2^4\right)\nonumber\\&=&\frac{1}{48}\left(13a_2^4-12a_2^2a_3-12a_3^2+12a_2a_4\right).\eea
Note that $H_{2,1}\left(F_{f^{-1}}/2\right)$ is invariant under rotation, since for $g_{\theta}(z)=e^{-i\theta}f\left(e^{i\theta}z\right)$, $\theta\in\mathbb{R}$ and $f\in\mathcal{S}$, we have
\beas H_{2,1}\left(F_{g_{\theta}^{-1}}/2\right)= \frac{e^{4i\theta}}{48}\left(13a_2^4-12a_2^2a_3-12a_3^2+12a_2a_4\right)=e^{4i\theta}H_{2,1}\left(F_{f^{-1}}/2\right).\eeas
Let $\mathcal{P}$ denote the class of all analytic functions $p$ with $p(0)=1$ and $\text{Re}(p(z))>0$ in $\mathbb{D}$, and is of the form
\bea\label{eq4} p(z)=1+c_1z+c_2z^2+\cdots.\eea  
A member of $\mathcal{P}$ is called a Carath\'eodory function. It is well-known that (see \cite[p. 41]{8}) $|c_n|\leq 2$ ($n\in\mathbb{N}$) for a function $p\in\mathcal{P}$. 
The main aim of this paper is to find the sharp upper bound of $H_{2,1}\left(F_{f^{-1}}/2\right)$ for functions $f$ in the class $\mathcal{P}(M)$.
\section{key lemmas}
Our computing is based on the well-known formula on coefficient $c_2$ (e.g., \cite[p. 166]{R1}) and on the formula $c_3$ due to Libera and
Zlotkiewicz \cite{17} and \cite{18}, both with $c_1 \geq 0$. The version below comes from \cite{R2}, where the extremal functions have been determined.
\begin{lem}\label{lem1} If $p\in\mathcal{P}$ is of the form (\ref{eq4}) with $c_1\geq 0$, then
\bea&&\label{l1} c_1=2p_1,\;c_2=2p_1^2+2(1-p_1^2)p_2,\\
&&\label{l2}c_3=2p_1^3+4(1-p_1^2)p_1p_2-2(1-p_1^2)p_1p_2^2+2(1-p_1^2)(1-|p_2|^2)p_3\eea
for some $p_1\in[0,1]$ and $p_2,p_3\in\ol{\mathbb{D}}:=\{z\in\mathbb{C} : |z|\leq 1\}$.\\
For $p_1\in\mathbb{T}:=\{z\in\mathbb{C} : |z|=1\}$, there is a unique function $p\in\mathcal{P}$ with $c_1$ as in (\ref{l1}), namely 
\beas p(z)=\frac{1+p_1z}{1-p_1z}, \;z\in\mathbb{D}.\eeas
For $p_1\in\mathbb{D}$ and $p_2\in\mathbb{T}$, there is a unique function $p\in\mathcal{P}$ with $c_1, c_2$ as in (\ref{l1}), namely 
\bea\label{er1} p(z)=\frac{1+(p_1+\ol{p_1}p_2)z+p_2z^2}{1-(p_1-\ol{p_1}p_2)z-p_2z^2}, \;z\in\mathbb{D}.\eea
For $p_1,p_2\in\mathbb{D}$ and $p_3\in\mathbb{T}$, there is a unique function $p\in\mathcal{P}$ with $c_1, c_2,c_3$ as in (\ref{l1}) and (\ref{l2}), namely 
\beas p(z)=\frac{1+(\ol{p_2}p_3+\ol{p_1}p_2+p_1)z+(\ol{p_1}p_3+p_1\ol{p_2}p_3+p_2)z^2+p_3z^3}{1+(\ol{p_2}p_3+\ol{p_1}p_2-p_1)z+(\ol{p_1}p_3-p_1\ol{p_2}p_3-p_2)z^2-p_3z^3}, \;z\in\mathbb{D}.\eeas
\end{lem}
\begin{lem}\cite{5}\label{lem2} Let $A,B,C\in\mathbb{R}$ and 
\beas Y(A,B,C):= \max_{z\in\ol{\mathbb{D}}}\left(\left|A+Bz+Cz^2\right|+1-|z|^2\right).\eeas
\item[(i)] If $AC\geq0$, then 
\beas Y(A,B,C)=\left\{\begin{array}{lll}
|A|+|B|+|C|,&|B|\geq 2\left(1-|C|\right),\\
1+|A|+\frac{B^2}{4\left(1-|C|\right)}, &|B|<2\left(1-|C|\right).
\end{array}\right.\eeas
\item[(ii)] If $AC<0$, then 
\beas Y(A,B,C)=\left\{\begin{array}{lll}
1-|A|+\frac{B^2}{4\left(1-|C|\right)},&-4AC\left(C^{-2}-1\right)\leq B^2\wedge |B|<2\left(1-|C|\right),\\
1+|A|+\frac{B^2}{4\left(1+|C|\right)}, &B^2<\min\left\{4\left(1+|C|\right)^2,-4AC\left(C^{-2}-1\right)\right\},\\
R(A,B,C),&\text{Otherwise},
\end{array}\right.\eeas
where
\beas R(A,B,C)=\left\{\begin{array}{lll}
|A|+|B|+|C|,&|C|\left(|B|+4|A|\right)\leq |AB|,\\
-|A|+|B|+|C|, &|AB|\leq |C|\left(|B|-4|A|\right),\\
\left(|A|+|C|\right)\sqrt{1-\frac{B^2}{4AC}},&\text{Otherwise}.
\end{array}\right.\eeas 
\end{lem}
\section{second Hankel determinant for logarithmic inverse coefficients} 
\noindent In the following result, we obtain the sharp bound for the second Hankel determinant of logarithmic inverse coefficients for functions in the class $\mathcal{P}(M)$.
\begin{theo}\label{Th1}
Let $f\in\mathcal{P}(M)$ with $0<M\leq 1/\log4 (\approx 0.7213475)$. Then 
\bea\label{fe1}  \left|H_{2,1}\left(F_{f^{-1}}/2\right)\right|\leq\left\{\begin{array}{ll}
M^2/36,&0<M\leq \dfrac{1}{39}(6+\sqrt{114})\\[3mm]
\dfrac{M^2}{144}\left(39M^2-12M+2\right),&\dfrac{1}{39}(6+\sqrt{114})<M\leq 1/\log4.\end{array}\right.\eea
The inequality (\ref{fe1}) is sharp.\end{theo}
\begin{proof} Fix $0<M\leq 1/\log4$ and let $f\in\mathcal{P}(M)$ be of the form (\ref{eq1}). Then $\text{Re}(zf'')>-M$ for $z\in\mathbb{D}$. Thus, it follows that 
\bea\label{e1} zf''(z)=Mp(z)-M,\;z\in\mathbb{D}\eea
for some $p\in\mathcal{P}$ of the form (\ref{eq4}).
Since the class $\mathcal{P}$ and $H_{2,1}\left(F_{f^{-1}}/2\right)$ are invariant under rotations, we may assume that $c_1\in [0, 2]$ \cite[p. 80, \textrm{Theorem} 3]{12}. 
With the help of \textrm{Lemma \ref{lem1}}, we have $p_1\in[0, 1]$ and $p_2,p_3\in\ol{\mathbb{D}}$. Using (\ref{eq1}), (\ref{eq4}) and (\ref{e1}), we deduce that 
\bea\label{e2} a_2=\frac{M}{2}c_1, a_3=\frac{M}{6}c_2, a_4=\frac{M}{12}c_3.\eea
In view of \textrm{Lemma \ref{lem1}} and from (\ref{eq3}) and (\ref{e2}), we have 
\beas\label{e3} H_{2,1}\left(F_{f^{-1}}/2\right)&=&\frac{1}{48}\left(13a_2^4-12a_2^2a_3-12a_3^2+12a_2a_4\right)\\
&=&\frac{1}{48}\left(\left(13M^4-4M^3+\frac{2M^2}{3}\right)p_1^4-\frac{2}{3}M^2(1-p_1^2)\left(2+p_1^2\right)p_2^2\right.\nonumber\\
&&\left.+\left(\frac{4M^2}{3}-4M^3\right)p_1^2p_2(1-p_1^2)+2M^2p_1p_3(1-p_1^2)(1-|p_2|^2)\right).\eeas
Now, we will discuss the following cases involving $p_1$.\\
{\bf Case 1.} Suppose that $p_1=1$. Then for $M>0$, from (\ref{e3}), we have 
\bea\label{ee1} \left|H_{2,1}\left(F_{f^{-1}}/2\right)\right|=\frac{1}{144}\left(39M^4-12M^3+2M^2\right).\eea
{\bf Case 2.} Suppose that $p_1=0$. Then for $M>0$, from (\ref{e3}), we have 
\bea\label{ee2} \left|H_{2,1}\left(F_{f^{-1}}/2\right)\right|=\frac{M^2}{36}|p_2|^2\leq \frac{M^2}{36}.\eea
{\bf Case 3.} Suppose that $p_1\in(0,1)$. Since $|p_3|\leq 1$, so by applying the triangle inequality in (\ref{e3}), we obtain
\bea\label{ee3} \left|H_{2,1}\left(F_{f^{-1}}/2\right)\right|&\leq& \frac{1}{48}\left|\left(13M^4-4M^3+\frac{2M^2}{3}\right)p_1^4-\frac{2}{3}M^2(1-p_1^2)\left(2+p_1^2\right)p_2^2\right.\nonumber\\
&&\left.+\left(\frac{4M^2}{3}-4M^3\right)p_1^2p_2(1-p_1^2)\right|+\frac{1}{48}\left|2M^2p_1(1-p_1^2)(1-|p_2|^2)\right|\nonumber\\
&=&\frac{M^2p_1(1-p_1^2)}{24}\left(|A+Bp_2+Cp_2^2|+1-|p_2|^2\right),\eea
where 
\bs\beas A=\frac{(39M^2-12M+2)p_1^3}{6(1-p_1^2)},\;B=\left(\frac{2}{3}-2M\right)p_1=\left\{\begin{array}{lll}\geq0&\text{for} \;M\leq 1/3\\<0&\text{for}\; M> 1/3\end{array}\right.\;\text{and}\; C=-\frac{(2+p_1^2)}{3p_1}.\eeas\es
Since $39M^2-12M+2>0$ for all $M>0$, so it is easy to observe that $AC<0$. In view of \textrm{Lemma \ref{lem2}}, we will discuss the following circumstances.\\[2mm]
{\bf Sub-case 3.1.} Note that the inequality $-4AC\left(C^{-2}-1\right)\leq B^2$ implies that
\bs\bea\label{e4} &&\frac{(39M^2-12M+2)p_1^3}{6(1-p_1^2)}\frac{(2+p_1^2)}{3p_1}\left(\frac{9p_1^2}{(2+p_1^2)^2}-1\right)-\frac{(2-6M)^2}{9}p_1^2\leq0\nonumber\\\text{{\it i.e.,}}
&&\frac{2p_1^2\left[21M^2p_1^4-(213M^2-72M+12)p_1^2+(192M^2-72M+12)\right]}{9(1-p_1^2)(2+p_1^2)}\geq 0, \eea\es
which is hold for $M>0$ and $p_1\in(0,1)$. However, for $0<M\leq 1/3$, the inequality $|B|<2\left(1-|C|\right)$ is equivalent to 
\beas \frac{(2-6M)}{3}p_1-2\left(1-\frac{(2+p_1^2)}{3p_1}\right)<0,
\;\text{{\it i.e.,}}\; 2(2-3M)p_1^2-6p_1+4<0,\eeas
which is true if, and only if, $p_1>0$ and $M > (2-3p_1+ 2p_1^2)/(3p_1^2)$. It is easy to see that $(2-3p_1+ 2p_1^2)/(3p_1^2)>1/3$  for $0<p_1<1$, which contradicts the fact that $0<M\leq 1/3$, as illustrated in {\bf Figure} \ref{fig1}.
Again, for $1/3<M\leq 1/\log4$, the inequality $|B|<2\left(1-|C|\right)$ implies that
\beas \frac{(6M-2)}{3}p_1-2\left(1-\frac{(2+p_1^2)}{3p_1}\right)<0,
\;\text{{\it i.e.,}}\;6Mp_1^2-6p_1+4<0,\eeas
which is true if, and only if, $p_1 > 0$ and $M<(-2+3p_1)/(3p_1^2)$. It is easy to see that $(-2+3p_1)/(3p_1^2)<1/3$ for $0 < p_1 < 1$, which contradicts the fact that $M<(-2+3p_1)/(3p_1^2)$ with $1/3<M\leq 1/\log4$ and $0 < p_1 < 1$, and it's shown in {\bf Figure} \ref{fig2}.
\begin{figure}[H]
\begin{minipage}[c]{0.5\linewidth}
\centering
\includegraphics[scale=0.7]{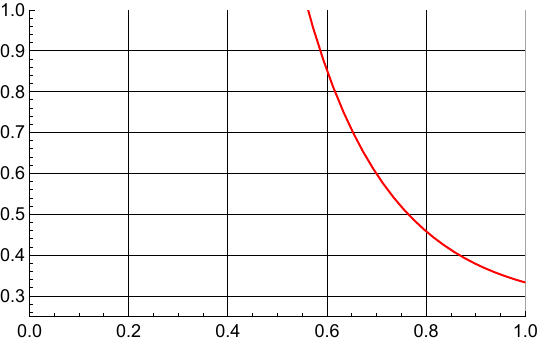}
\caption{The graph of $(2-3p_1+ 2p_1^2)/(3p_1^2)$ for $p_1\in\left(0,1\right)$}
\label{fig1}
\end{minipage}
\hfill
\begin{minipage}[c]{0.49\linewidth}
\centering
\includegraphics[scale=0.7]{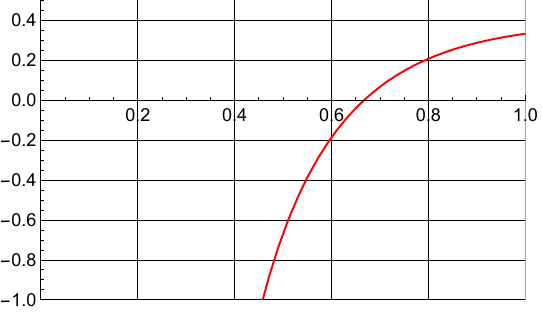}
\caption{The graph of $(-2+3p_1)/(3p_1^2)$ for $p_1\in\left(0,1\right)$}
\label{fig2}
\end{minipage}
\end{figure}
\noindent{\bf Sub-case 3.2.} Note that 
\beas &&4\left(1+|C|\right)^2=4\left(1+\frac{(2+p_1^2)}{3p_1}\right)^2=\frac{4(p_1^4+6p_1^3+13p_1^2+12p_1+4)}{9p_1^2}>0\\\text{and}
&&-4AC\left(C^{-2}-1\right)=-\frac{2p_1^2}{9}\cdot\frac{(39M^2-12M+2)(p_1^4-5p_1^2+4)}{(1-p_1^2)(2+p_1^2)}<0.\eeas
Therefore, $\min\left\{4\left(1+|C|\right)^2,-4AC\left(C^{-2}-1\right)\right\}=-4AC\left(C^{-2}-1\right)$. Also from (\ref{e4}), we know that 
\beas-4AC\left(C^{-2}-1\right)\leq B^2\;\text{hold for}\;M>0\;\text{and}\; p_1\in(0,1).\eeas
Therefore the inequality 
\beas B^2<\min\left\{4\left(1+|C|\right)^2,-4AC\left(C^{-2}-1\right)\right\}=-4AC\left(C^{-2}-1\right)\eeas
does not hold for $M>0$ and $p_1\in(0,1)$.\\
{\bf Sub-case 3.3.}  Corresponding to the values of $M$, we consider the following cases.\\
{\bf Sub-case 3.3.1} For $0<M\leq 1/3$, the inequality $|C|\left(|B|+4|A|\right)-|AB|\leq 0$ is equivalent to
\beas
(117M^3+3M^2)p_1^4+6(26M^2-7M+1)p_1^2+4(1-3M)\leq 0,\eeas 
is not true for $p_1\in(0,1)$, since $26M^2-7M+1>0$ for $M\in(0, 1/3]$. \\
{\bf Sub-case 3.3.2} For $1/3<M\leq 1/\log4$, the inequality $|C|\left(|B|+4|A|\right)-|AB|\leq 0$ is equivalent to
\bea\label{N1} 
\Omega_1(p_1^2)\geq 0,\eea
where $\Omega_1(t)=(117M^3-153M^2+48M-8)t^2+(-156M^2+54M-10)t+4(1-3M)$ with $t\in(0,1)$.
Since $M\in(1/3,1/\log4]$, so $\Delta(M):=12 (19 - 186 M + 899 M^2 - 2172 M^3 + 2496 M^4)>0$. Now $\Omega_1(t)=0$ gives
\bea\label{MI1}\left\{\begin{array}{lll} t_1=\dfrac{78M^2-27M+5- \sqrt{3(19 - 186 M + 899 M^2 - 2172 M^3 + 2496 M^4)}}{117M^3-153M^2+48M-8}\\[3mm]
 t_2=\dfrac{78M^2-27M+5+\sqrt{3(19 - 186 M + 899 M^2 - 2172 M^3 + 2496 M^4)}}{117M^3-153M^2+48M-8}.\end{array}\right.\eea
Note that $117M^3-153M^2+48M-8<0$ and $78M^2-27M+5>0$ for $M\in(1/3,1/\log4]$.
It is easy to see that $t_2<0$ and we also claim that $t_1<0$. As a matter of fact that the inequality $t_1 < 0$ is equivalent to the inequality
\beas \Psi(M):=351 M^4 - 576 M^3 + 297 M^2 - 72 M + 8<0,\eeas
which are true for $1/3<M\leq 1/\log4$, as illustrated in {\bf Figure} \ref{fig3}. Thus, the inequality (\ref{N1}) is false.
\begin{figure}[H]
\includegraphics[scale=0.8]{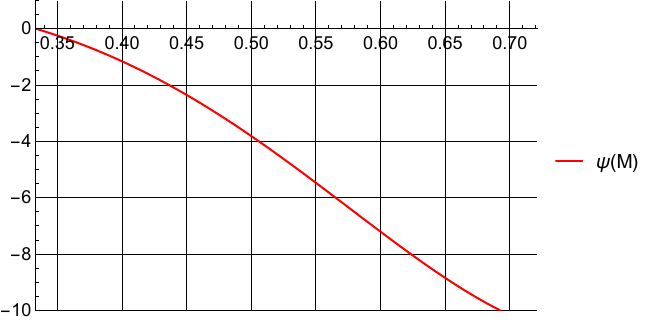}
\caption{The graph of the polynomial $\Psi(M)$ for $1/3<M\leq 1/\log4$}
\label{fig3}
\end{figure}
\noindent {\bf Sub-case 3.4.}  Corresponding to the values of $M$, we consider the following cases.
{\bf Sub-case 3.4.1}
For $0<M<1/3$, the inequality $|AB|-|C|\left(|B|-4|A|\right)\leq 0$ is equivalent to the inequality (\ref{N1}).
By using similar arguments to those of \textrm{Sub-case 3.3}, we get (\ref{MI1}) with $\Delta(M):=12 (19 - 186 M + 899 M^2 - 2172 M^3 + 2496 M^4)>0$, $117M^3-153M^2+48M-8<0$ and $78M^2-27M+5>0$ for $M\in(0,1/3)$.
It is easy to see that $t_2<0$ and we also claim that $0<t_1<1$. As a matter of fact that both the inequality $t_1 >0$ and $t_1<1$ are respectively equivalent to the inequalities
\beas \Psi_1(M)>0\;\text{and}\;\Psi_2(M)>0,\eeas
which are true for $M\in(0,1/3)$, where
\beas &&\Psi_1(M)=351 M^4 - 576 M^3 + 297 M^2 - 72 M + 8\;\text{and} \\
&&\Psi_2(M)=13689 M^6 - 54054 M^5 + 63423 M^4 - 31176 M^3 + 8934 M^2 -1392 M + 112\eeas
and it's shown in {\bf Figure} \ref{fig4}.
\begin{figure}[H]
\includegraphics[scale=0.95]{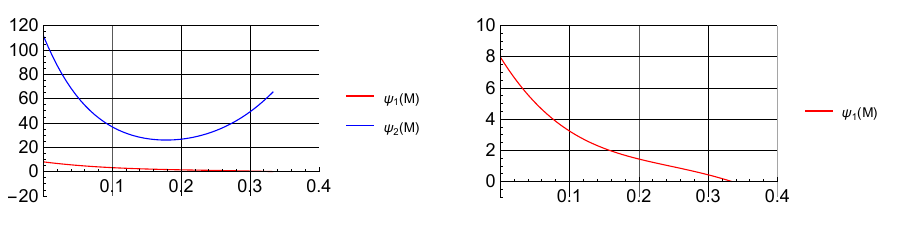}
\caption{The graph of the polynomials $\Psi_1(M)$ and $\Psi_2(M)$ for $0<M<1/3$}
\label{fig4}
\end{figure}
Thus, the inequality (\ref{N1}) is valid for $\sqrt{t_1}\leq p_1<1$ whenever $M\in(0,1/3)$.  In view of \textrm{Lemma \ref{lem2}} and (\ref{ee3}), we have
\bea\label{M2} \left|H_{2,1}\left(F_{f^{-1}}/2\right)\right|&\leq& \frac{M^2p_1(1-p_1^2)}{24}\left(-|A|+|B|+|C|\right)\nonumber\\[2mm]
&=&\frac{M^2}{144}\left[4 +(2-12M)p_1^2-\left(8-24M+39 M^2\right)p_1^4\right].\eea
Let $\Phi_1(x)=-\left(39 M^2-24M+8\right)x^2 +2(1-6M)x+4$, where $x\in[t_1,1)$. Then 
\beas &&\Phi'_1(x)=-2\left(39 M^2-24M+8\right)x +2(1-6M), \Phi''_1(x)=-2\left(39 M^2-24M+8\right)<0.\eeas
It is clear that the function $\Phi_1(x)$ is decreasing for $1/6\leq M<1/3$ and thus, we have  
\beas\Phi_1(x)\leq \Phi_1(t_1)\;\text{for}\;x\in[t_1,1).\eeas
For $0<M< 1/6$, the function $\Phi_1(x)$ has a unique critical point which is $y_0=(1-6M)/(39 M^2-24M+8)$ and the function $\Phi_1(x)$ is increasing (resp. decreasing) according as $x<y_0$ (resp. $x>y_0$). Since $39 M^2-24M+8>$ for all $M\in(0,1/3)$, so $y_0>0$ for $0<M< 1/6$. It remains to check whether $t_1\leq y_0<1$.
The inequality $y_0\geq t_1$ is equivalent to 
\beas\phi_1(M):&=&876096 M^8 - 1908036 M^7 + 1561005 M^6 - 584694 M^5 + 24093 M^4 \\
&&+ 66768 M^3 - 28496 M^2 + 5376 M - 448\geq 0,\eeas
which is not true for $M\in(0,1/6)$, as illustrated in {\bf Figure} \ref{fig5}.
\noindent Thus $\Phi_1(x)$ is decreasing for $M\in(0,1/6)$. For $x\in[t_1,1)$, we have
\beas\Phi_1(x)\leq \Phi_1(t_1)&\text{for}\;0<M< 1/3.\eeas
From (\ref{M2}) and for $x\in[t_1,1)$, we have
\bs\bea\label{M6}\left|H_{2,1}\left(F_{f^{-1}}/2\right)\right|\leq \frac{M^2\left(\psi_1(M)+\psi_2(M)\sqrt{3(19 -186 M+899 M^2-2172M^3 +2496 M^4)}\right)}{6(117M^3-153M^2+48M-8)^2}\nonumber\\\eea\es 
for $0<M< 1/3$, where 
\bs\bea\label{rm2}\left\{\begin{array}{ll}
\psi_1(M)=-24336 M^6+33345 M^5-21801 M^4+8415 M^3-2023 M^2+288 M-20,\\[2mm]
\psi_2(M)=312M^4-330M^3+159M^2-36M+4.\end{array}\right.\eea\es
{\bf Sub-case 3.4.2}
For $1/3\leq M\leq 1/\log4$, the inequality $|AB|-|C|\left(|B|-4|A|\right)\leq 0$ is equivalent to the inequality 
\bea\label{M3}
\Omega_2(p_1^2)\leq 0,\eea
where $\Omega_2(t)=(117M^3+3M^2)t^2+2(78M^2-21M+3)t+4(1-3M)$ with $t\in(0,1)$.
Since $M\in[1/3,1/\log4]$, so $\Delta(M):=12 (3 - 42 M + 299 M^2 - 1236 M^3 + 2496 M^4)>0$. Now $\Omega_2(t)=0$ gives
\bea\label{M1}\left\{\begin{array}{lll} t_3=\dfrac{-78M^2+21M-3 -\sqrt{3(3 - 42 M + 299 M^2 - 1236 M^3 + 2496 M^4)}}{117M^3+3M^2}\\[3mm]
 t_4=\dfrac{-78M^2+21M-3 +\sqrt{3(3 - 42 M + 299 M^2 - 1236 M^3 + 2496 M^4)}}{117M^3+3M^2}.\end{array}\right.\eea
Note that $117M^3+3M^2>0$ and $78M^2-21M+3>0$ for $M\in[1/3,1/\log4]$.
It is easy to see that $t_3<0$ and we also claim that $0<t_4<1$. As a matter of fact that both the inequality $t_4>0$ and $t_4<1$ are respectively equivalent to the inequalities
\beas \Psi_3(M)>0\;\text{and}\;\Psi_4(M)>0,\eeas
which are true for $M\in(1/3,1/\log4]$, where
\beas &&\Psi_3(M)= 1404 M^2- 432 M -12\;\text{and} \\
&&\Psi_4(M)= 13689 M^4+ 18954 M^3- 5841 M^2+ 1008 M +30,\eeas
as illustrated in {\bf Figure} \ref{fig6}.
At $M=1/3$, the inequality (\ref{M3}) becomes $14 p_1^2(p_1^2+2)/3\leq 0$, which is not possible for $p_1\in(0,1)$. 
\begin{figure}[H]
\begin{minipage}[c]{0.5\linewidth}
\includegraphics[scale=0.65]{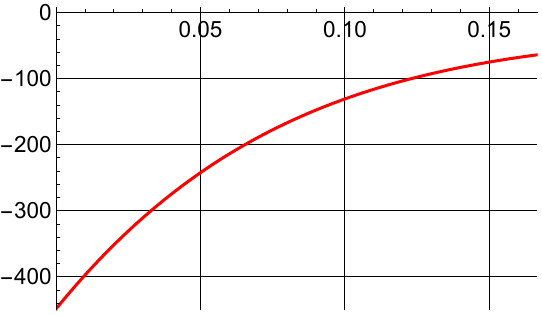}
\caption{The graph of the polynomial $\phi_1(M)$ for $0<M< 1/6$}
\label{fig5}
\end{minipage}
 \hfill
\begin{minipage}[c]{0.49\linewidth}
\includegraphics[scale=0.65]{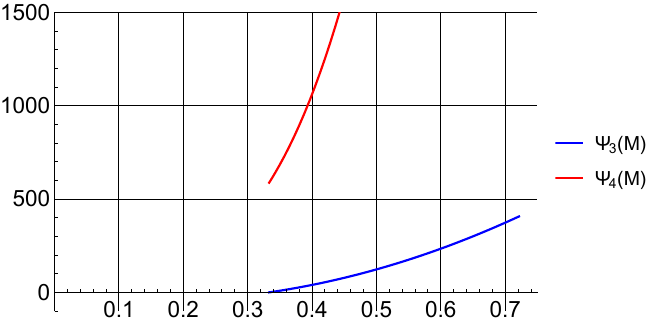}
\caption{The graph of the polynomials $\Psi_3(M)$ and $\Psi_4(M)$ respectively within $1/3<M\leq 1/\log4$}
\label{fig6}
\end{minipage}
\end{figure}
\noindent Thus, the inequality (\ref{M3}) is valid for $0<p_1\leq \sqrt{t_4}$ whenever $M\in(1/3,1/\log4]$.  In view of \textrm{Lemma \ref{lem2}} and (\ref{ee3}), we have
\bea\label{M4} \left|H_{2,1}\left(F_{f^{-1}}/2\right)\right|&\leq& \frac{M^2p_1(1-p_1^2)}{24}\left(-|A|+|B|+|C|\right)\nonumber\\
&=&\frac{M^2}{144}\left[4 -6(1-2M)p_1^2-39 M^2p_1^4\right].\eea
Let $\Phi_2(x)=-39 M^2x^2 -6(1-2M)x+4$, where $x\in(0,t_4]$. Then 
\beas &&\Phi'_2(x)=-78M^2x -6(1-2M), \Phi''_2(x)=-78M^2<0.\eeas
It is clear that the function $\Phi_2(x)$ is decreasing for $1/3<M\leq 1/2$ and thus, we have  
\beas\Phi_2(x)\leq \Phi_2(0)=4\;\text{for}\;x\in(0,t_4].\eeas
For $1/2<M\leq 1/\log4$, the function $\Phi_2(x)$ has a unique critical point which is $y_1=(2M-1)/(13M^2)$ and the function $\Phi_2(x)$ is increasing (resp. decreasing) according as $x<y_1$ (resp. $x>y_1$). Since $M\in(1/2,1/\log4]$, so $y_1>0$. It remains to check whether $y_1\leq t_4$.
The inequality $y_1\leq t_4$ is equivalent to the inequality 
\beas\phi_2(M):= - 292032 M^4+ 331812 M^3 - 85719 M^2 + 6354 M+225\geq 0,\eeas
which is true for $M\in(1/2,1/\log4]$, as illustrated in {\bf Figure} \ref{fig7}.
\noindent For $x\in(0,t_4]$, we have
\beas\Phi_2(x)\leq\left\{\begin{array}{lll}
 4&\text{for}\;1/3<M\leq 1/2,\\[2mm]
 \Phi_2(y_1)=(64 M^2 - 12 M +3 )/(13 M^2)&\text{for}\;1/2<M\leq 1/\log4.\end{array}\right.\eeas
From (\ref{M4}) and for $x\in(0,t_4]$, we have
\bea\label{MI6}\left|H_{2,1}\left(F_{f^{-1}}/2\right)\right|\leq \left\{\begin{array}{lll}
 M^2/36&\text{for}\;1/3<M\leq 1/2,\\[2mm]
 (64 M^2 - 12 M +3 )/1872&\text{for}\;1/2<M\leq 1/\log4.\end{array}\right.\eea
 {\bf Sub-case 3.4.3} We now consider the case $$p_1\in\left\{\begin{array}{lll}
 (0,\sqrt{t_1}),&\text{when}\;M\in(0,\frac{1}{3}),\\[2mm]
  (\sqrt{t_4},1),&\text{when}\;M\in[1/3,1/\log4],\end{array}\right.$$
 where $t_1$ and $t_4$ are given in (\ref{MI1}) and (\ref{M1}) respectively. In view of \textrm{Lemma \ref{lem2}} and (\ref{ee3}), we have 
\bs\bea\label{e7}\left|H_{2,1}\left(F_{f^{-1}}/2\right)\right|&\leq& \frac{M^2p_1(1-p_1^2)}{24}\left(|A|+|C|\right)\sqrt{1-\frac{B^2}{4AC}}\nonumber\\[2mm]
&=&\frac{M^2}{144}\left\{(39M^2-12M)p_1^4-2p_1^2+4\right\}\sqrt{\frac{21M^2p_1^2+6(16M^2-6M+1)}{(39M^2-12M+2)(2+p_1^2)}}.\nonumber\\\eea\es
Let $\xi(x)=\xi_1(x)\sqrt{\xi_2(x)}$, where 
\beas x\in\left\{\begin{array}{lll}
 (0,t_1)&\text{for}\;M\in(0,1/3),\\[2mm]
  (t_4,1)&\text{for}\;M\in[1/3,1/\log4]\end{array}\right.\eeas
  with
 \beas \xi_1(x)=(39M^2-12M)x^2-2x+4\quad\text{and}\quad \xi_2(x)=\frac{21M^2x+6(16M^2-6M+1)}{(39M^2-12M+2)(2+x)}.\eeas 
It is evident that  
\be\label{Q1} \xi'_1(x)=2(39M^2-12M)x-2\quad\text{and}\quad\xi_2'(x)=\frac{-6 (9 M^2- 6 M+1)}{(39 M^2 - 12 M+2) (2 + x)^2}\leq 0.\ee
We now consider the following cases corresponding to the values of $M$.\\
{\bf Sub-case 3.4.3.1} Suppose $0<M< 1/3$. From (\ref{Q1}), it is evident that $\xi_2(x)$, $x\in(0,t_1)$ is decreasing and 
 $\xi_1(x)$, $x\in(0,t_1)$ is decreasing for $0<M\leq M_1$, where $M_1=4/13$ is the unique positive real root of the equation $39M^2-12M=0$. Now, 
for $4/13<M<1/3$, the function $\xi_1(x)$ has a unique critical point $y_3=1/(39M^2-12M)>0$ and the function $\xi_1(x)$, $x\in(0,t_1)$ is decreasing (resp. increasing) according 
as $x<y_3$ (resp. $x>y_3$). It remains to check whether $y_3<t_1$ for $4/13<M<1/3$.
The inequality $y_3<t_1$ is equivalent to $\phi_3(M):=2135484 M^8 - 4106700 M^7 + 2755701 M^6 - 823878 M^5 + 42201 M^4 + 45144 M^3 - 14208 M^2 + 1728 M - 64>0$, which is not true $4/13<M<1/3$ and it's shown in {\bf Figure} \ref{fig8}.
\begin{figure}[H]
\begin{minipage}[c]{0.5\linewidth}
\centering
\includegraphics[scale=0.7]{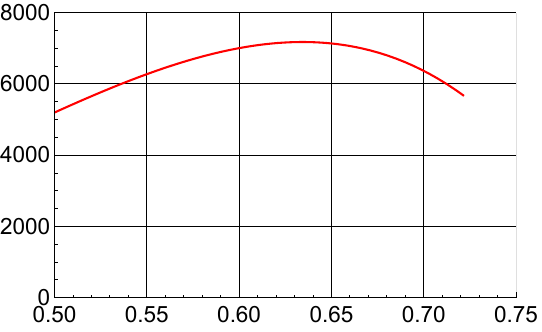}
\caption{The graph of the polynomial $\phi_2(M)$ for $1/2<M\leq 1/\log4$}
\label{fig7}
\end{minipage}
\hfill
\begin{minipage}[c]{0.49\linewidth}
\centering
\includegraphics[scale=0.7]{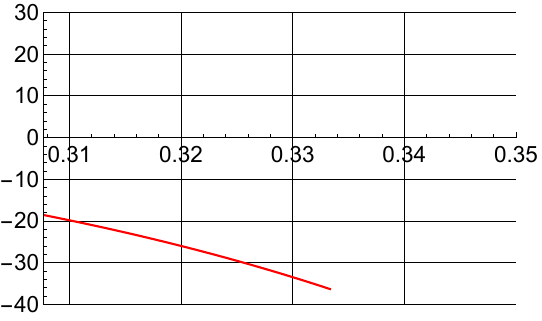}
\caption{The graph of the polynomial $\phi_3(M)$ for $4/13<M< 1/3$}
\label{fig8}
\end{minipage}
\end{figure}
\noindent Thus $\xi_1(x)$, $x\in(0,t_1)$ is decreasing for $4/13<M<1/3$. As a result, both the functions $\xi_1(x)$ and $\xi_2(x)$ are decreasing for $x\in(0,t_1)$ and $0<M< 1/3$. Hence, we derive from (\ref{e7}) that
\bea\label{Q2}\left|H_{2,1}\left(F_{f^{-1}}/2\right)\right|\leq \frac{M^2}{36}\sqrt{\frac{3(16M^2-6M+1)}{(39M^2-12M+2)}}\;\;\text{for}\;0<M< 1/3.\eea
{\bf Sub-case 3.4.3.2} Suppose $1/3\leq M\leq 1/\log4$. From (\ref{Q1}) it is easy to see that $\xi_2(x)$, $x\in(t_4,1)$ is a positive decreasing function.
We claim that $\xi(x)$ is a convex function, {\it i.e.,} we have to show that $\xi''(x)\geq 0$. Now 
\bs\beas\xi''(x)(\xi_2(x))^{3/2}&=&\xi_1''(x)(\xi_2(x))^2+\xi_1'(x)\xi_2(x)\xi_2'(x)+\frac{1}{2}\xi_1(x)\xi_2(x)\xi_2''(x)-\frac{1}{4}\xi_1(x)(\xi_2'(x))^2\\[2mm]
&=&\frac{9\left(x^4M^5\Psi_1(M) + x^3M^3\Psi_2(M)+x^2M \Psi_3(M) +x\Psi_4(M)+ \Psi_5(M)\right)}{(39 M^2-12M+2)^2 (2 + x)^4},\eeas\es
where 
$\Psi_1(x)=3822 M - 1176$, $\Psi_2(M)=45318 M^3 - 23772 M^2 + 4662 M - 504$, 
$\Psi_3(M)= 189657 M^5 - 129960 M^4 + 38178 M^3 - 6372 M^2 + 549 M - 36$, 
$\Psi_4(M)=369408 M^6 - 312096 M^5 + 117762 M^4 - 25464 M^3 + 3148 M^2 - 216 M + 2$, 
$\Psi_5(M)=319488 M^6 - 337920 M^5 + 162876 M^4 - 45456 M^3 + 7592 M^2 - 720 M + 28$. 

 It is easy to see that $\Psi_j(M)>0$ $(1\leq j\leq5)$ for all $M\in[1/3,1/\log4]$, as illustrated in in {\bf Figure} \ref{fig9}. Thus, we have $\xi''(x)\geq 0$ for $x\in(t_4,1)$.
 \begin{figure}[H]
\includegraphics[scale=0.8]{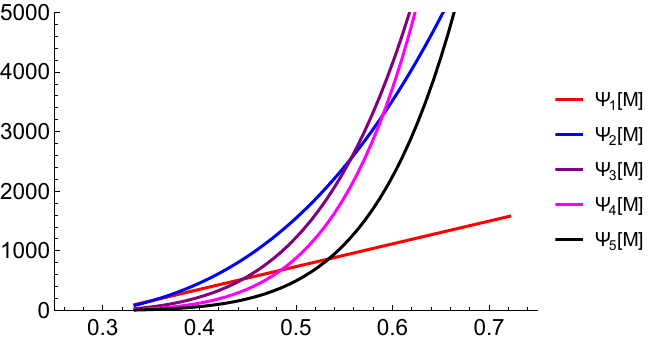}
\caption{The graph of the polynomials $\Psi_j(M)$ $(1\leq j\leq5)$ for $1/3\leq M\leq 1/\log4$}
\label{fig9}
\end{figure}

From (\ref{e7}), we have
\beas\left|H_{2,1}\left(F_{f^{-1}}/2\right)\right|\leq 
 \max\left\{\frac{M^2}{144}\xi_1(t_4)\sqrt{\xi_2(t_4)},\frac{M^2}{144}\xi_1(1)\sqrt{\xi_2(1)}\right\}.\eeas
 A tedious long calculation shows that
 \bs\beas &&\xi_1(t_4)=\dfrac{4(A_1-B_1\sqrt{3(3 - 42 M + 299 M^2 - 1236 M^3 + 2496 M^4)})}{3 M^3 (1 + 39 M)^2}\\[2mm]\text{and}
&& \xi_2(t_4)=\dfrac{3M^2(A_2+7\sqrt{3(3 - 42 M + 299 M^2 - 1236 M^3 + 2496 M^4)})}{(39M^2-12M+2)(B_2+\sqrt{3(3 - 42 M + 299 M^2 - 1236 M^3 + 2496 M^4)})},
 \eeas\es
 where 
\beas\left\{\begin{array}{lll}
A_1=48672 M^5 - 34515 M^4 + 12486 M^3 - 2577 M^2 + 312 M - 18,\\
B_1=507 M^3 - 273 M^2 + 62 M - 6,\\
A_2= 3744 M^3 - 1854 M^2+ 345 M-15 \;\;\text{and}\\
B_2= 234 M^3- 72 M^2+ 21 M-3.\end{array}\right.\eeas
It is clear that $\xi_1(1)\sqrt{\xi_2(1)}=(39M^2-12M+2)$.
Thus, we have 
\bea\label{Q3}\left|H_{2,1}\left(F_{f^{-1}}/2\right)\right|\leq\left\{\begin{array}{lll} 
 \frac{M^2}{144}\xi_1(t_4)\sqrt{\xi_2(t_4)}&\text{for}\;1/3\leq M\leq M_3,\\[2mm]
 \frac{M^2}{144}(39M^2-12+2)&\text{for}\;M_3\leq M\leq 1/\log4,\end{array}\right.\eea
where $M_3(\approx 0.423458)$ is the unique positive root of the equation $\xi_1(t_4)\sqrt{\xi_2(t_4)}=39M^2-12M+2$, as illustrated in {\bf Figure} \ref{fig10}.
\begin{figure}[H]
\includegraphics[scale=0.7]{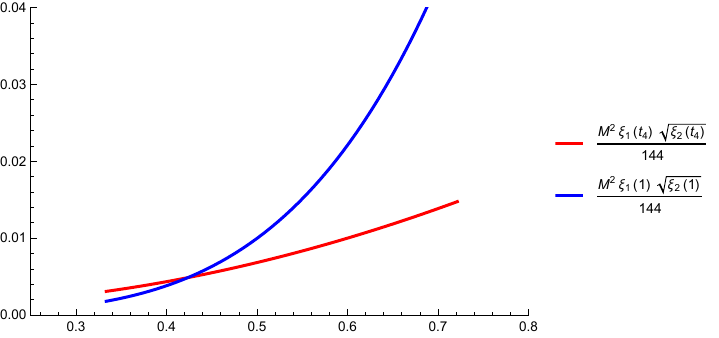}
\caption{The graph of the polynomials $M^2\xi_1(t_4)\sqrt{\xi_2(t_4)}/144$ and $M^2\xi_1(1)\sqrt{\xi_2(1)}/144$ for $1/3\leq M\leq 1/\log4$}
\label{fig10}
\end{figure}
\noindent{\bf Case 4.} It is evident that
\item[(I)] The inequality $(M^2/144)\left(39M^2-12M+2\right)\leq (M^2/36)$ is true for $0<M\leq (1/39)(6+\sqrt{114})\approx 0.427617$;\\
\item[(II)] The inequality $(M^2/36)\leq(M^2/144)\left(39M^2-12M+2\right)$ is true for $M\geq (1/39)(6+\sqrt{114})$;\\
\item[(III)] The inequality$$\frac{M^2}{36}\sqrt{\frac{3(16M^2-6M+1)}{(39M^2-12M+2)}}\leq \frac{M^2}{36}~\text{is true for all $M>0$};$$
\item[(IV)] The inequality $$\frac{M^2\left(\psi_1(M)+\psi_2(M)\sqrt{3(19 -186 M+899 M^2-2172M^3 +2496 M^4)}\right)}{6(117M^3-153M^2+48M-8)^2}\leq \frac{M^2}{36}$$ is equivalent to $\Phi_5(M)\geq 0$,
which is true for $0<M\leq 1/3$, where $\Phi_5(M)=-735140367 M^{12} + 3004133184 M^{11} - 5375600802 M^{10} + 
 5646621132 M^9 - 3923336331 M^8 + 1908662292 M^7 - 666386676 M^6 + 166905792 M^5 - 29086704 M^4 + 3220992 M^3 - 162816 M^2 - 6144 M + 1024$ and it's shown in {\bf Figure} \ref{fig11}.
 \begin{figure}[H]
\begin{minipage}[c]{0.5\linewidth}
\centering
\includegraphics[scale=0.7]{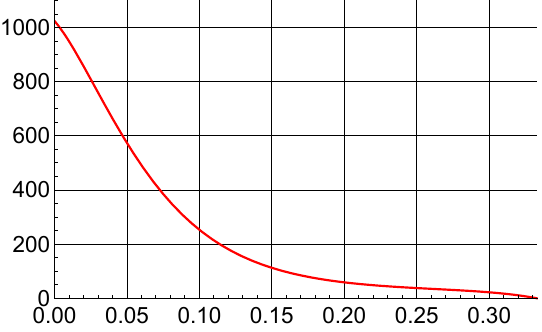}
\caption{The graph of the polynomial $\Phi_5(M)$ for $0<M\leq 1/3$}
\label{fig11}
\end{minipage}
\hfill
\begin{minipage}[c]{0.49\linewidth}
\centering
\includegraphics[scale=0.7]{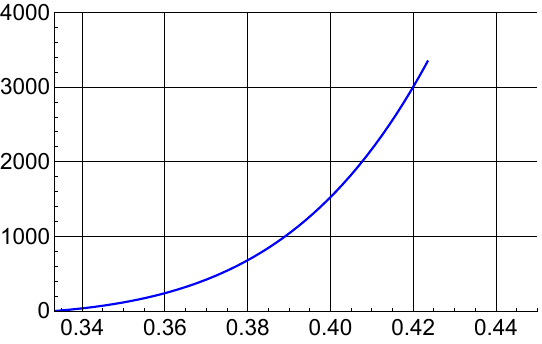}
\caption{The graph of the polynomial $\Phi_8(M)$ for $1/3\leq M\leq M_3$}
\label{fig12}
\end{minipage}
\end{figure}
\item[(V)] The inequality $(64M^2-12M+3)/1872\leq M^2\left(39M^2-12M+2\right)/144$ is equivalent to $507 M^4 - 156 M^3 - 38 M^2 + 12 M - 3\geq 0$, 
which is true for $1/2<M\leq 1/\log4$;
\item[(VI)] A tedious long calculation shows that the inequality
\beas \frac{M^2}{144}\xi_1(t_4)\sqrt{\xi_2(t_4)}\leq \frac{M^2}{36}\;\;\text{is equivalent to}\eeas 
$\Phi_8(M):=\Phi_6(M)+\Phi_7(M)\sqrt{3(2496 M^4-1236 M^3+299 M^2-42 M+3)}\geq 0$ which is true for $1/3\leq M\leq M_3(\approx 0.423458)$, as illustrated in {\bf Figure} \ref{fig12}, where 
\bs\beas&&\Phi_6(M)=-1779161054814 M^{13} + 3824833597416 M^{12} - 3931878351375 M^{11}\\
&& + 2565044468649 M^{10} - 1185433827318 M^9 + 409947682644 M^8- 109189509687 M^7\\
&&+ 22694555717 M^6 - 3684223958 M^5 +461874822 M^4 -43514160 M^3+ 2920752 M^2\\
&&- 125280 M + 2592\\[2mm]\text{and} 
&&\Phi_7(M)=20531017728  M^{11} - 39090726675  M^{10} + 35153704872  M^9 -19771596624  M^8\\
&& + 7732426260 M^7 - 2208185547  M^6 + 470004964  M^5 - 74606178  M^4 + 8663264 M^3\\
&& - 701712 M^2+ 35712  M - 864.\eeas\es
\noindent Utilizing (I)-(VI), we proceed to compare the bounds in (\ref{ee1}), (\ref{ee2}), (\ref{M6}), (\ref{MI6}), (\ref{Q2}) and (\ref{Q3}), which results in the following conclusion:
\bea\label{er2} \left|H_{2,1}\left(F_{f^{-1}}/2\right)\right|\leq\left\{\begin{array}{ll}
\dfrac{M^2}{36},&0<M\leq \dfrac{1}{39}(6+\sqrt{114})\\[2mm]
\dfrac{M^2}{144}\left(39M^2-12M+2\right),&\dfrac{1}{39}(6+\sqrt{114})<M\leq 1/\log4.\end{array}\right.\eea
In order to show that the inequalities in (\ref{er2}) are sharp. For $0<M\leq (6+\sqrt{114})/39$, in view of \textrm{Lemma \ref{lem1}}, we conclude that equality holds for the function $f\in\mathcal{A}$ given by (\ref{e1}), where 
$p\in\mathcal{P}$ is of the form (\ref{er1}) with $p_1=0$ and $p_2=-1$, {\it i.e.,}
\beas p(z)=\frac{1-z^2}{1+z^2}, \;z\in\mathbb{D}.\eeas
For $(6+\sqrt{114})/39<M\leq 1/\log4$, in view of \textrm{Lemma \ref{lem1}}, we conclude that equality holds for the function $f\in\mathcal{A}$ given by (\ref{e1}) with $p\in\mathcal{P}$ 
is of the form $p(z)=(1+z)(1-z), \;z\in\mathbb{D}$.
This completes the proof.
\end{proof}
\noindent{\bf Declarations}\\
\noindent{\bf Acknowledgment:} The work of the author is supported by  University Grants Commission (IN) fellowship (No. F. 44-1/2018 (SA-III)).\\
{\bf Conflict of Interest:} The author have no conflict of interest.\\
{\bf Data availability:} Not applicable.

\end{document}